\documentclass[a4paper,12pt]{amsart}
\usepackage{amsmath}
\usepackage[margin=1in]{geometry}
\usepackage{setspace}
\usepackage{amsfonts,amssymb,amsmath}
\newenvironment{poliabstract}[1]
  {\renewcommand{\abstractname}{#1}\begin{abstract}}
  {\end{abstract}}

\newtheorem{theorem}{Theorem}[section]
\newtheorem{lemma}[theorem]{Lemma}
\newtheorem{proposition}[theorem]{Proposition}

\theoremstyle{definition}
\newtheorem{definition}[theorem]{Definition}

\newcommand{\R}{\mathbb R} 
\newcommand{\N}{\mathbb N}  
\newcommand{\E}{\mathbb E}
\newcommand{\Pro}{\mathbb P}
\newcommand{\ve}{\varepsilon}
\newcommand{\txi}{\tilde{\xi}}
\newcommand{\Td}{\mathbb{T}^d}
\newcommand{\Tt}{\mathbb{T}^2}
\newcommand{\al}{\alpha}
\newcommand{\be}{\beta}

\newcommand{\W}{\bar{W}}

\begin{document}

\title{From a kinetic equation to
a  diffusion under an anomalous scaling}
\author[G. Basile]{Giada Basile}
 \address{G. Basile\\
Dipartimento di Matematica\\
Universit\`a di Roma La Sapienza\\
Piazzale Aldo Moro 5\\
00185 Roma, Italy     
} \email{basile@mat.uniroma1.it}



\date{\today}

\thanks{This research was supported in part  through the German Research
Council in the SFB 611}

\maketitle

\begin{poliabstract}{Abstract}
A linear Boltzmann equation is interpreted as
the forward equation for the probability density of a Markov process
$(K(t), i(t), Y(t))$ on $(\Tt\times\{1,2\}\times \R^2)$, where $\Tt$
is the two-dimensional torus. Here $(K(t), i(t))$ is an autonomous
reversible jump process, with waiting times between two jumps with
finite expectation value but infinite variance.  $Y(t)$ is an
additive functional of $K$, defined as $\int_0^t v(K(s))ds$, where
$|v|\sim 1$ for small $k$ .  We prove that the rescaled process
$(N\ln N)^{-1/2}Y(Nt)$ converges in distribution to a two-dimensional
Brownian motion.
As a consequence, the appropriately rescaled
solution of the Boltzmann equation converges to the solution of a diffusion equation.
\end{poliabstract}

\begin{poliabstract}{R\'esum\'e}
Une \'equation de Boltzmann lin\'eaire est interpr\'et\'ee comme
\'equation  de Fokker-Planck associ\'ee \`a la densit\'e de probabilit\'e
d'un processus de Markov $(K(t), i(t), Y(t))$ sur $(\Tt\times\{1,2\}\times
\R^2)$,
o\`u $\Tt$ est le tore bidimensionnel. Le processus Markovien $(K(t), i(t))$
est ici un processus de sauts r\'eversible avec des temps
d'attente entre deux sauts \`a moyenne finie mais variance infinie.
$Y(t)$ est une fonctionnelle additive de $K$, d\'efinie par $Y(t) = \int_0^t
v(K(s))ds$,
o\`u $|v|\sim 1$ pour $k$ petit.
Nous prouvons que le processus  $(N\ln N)^{-1/2}Y(Nt)$
converge en distribution vers un mouvement brownien bidimensionnel.
En cons\'equence, et moyennant un changement d'\'echelle appropri\'e, la
solution  de l'\'equation de Boltzmann converge vers celle d' une \'equation de
diffusion.

\end{poliabstract}

\section{Introduction}

One of the most interesting aspects  of the problem of energy transport in a solid
is an
anomalous  thermal conduction observed in low dimensional
materials (see \cite{LLP}, \cite{Dh} for a general review; see also
\cite{GBN} for experimental data for graphene materials). So far 
very few results are obtained by a rigorous analysis of microscopic
dynamics, and even crucial points, such as the exponent of the  divergence of
thermal conductivity in dimension one, are still debated.

 The theoretical approach proposed by
Peierls 
\cite{Pei} 
intended  to compute
thermal conductivity in analogy with the kinetic theory of gases, 
%
%
conforming to the idea  that 
at low temperatures the lattice vibrations,
responsible of energy transport, can be described as a gas of
interacting particles (phonons).
 The time-dependent distribution function of phonons
 solves a Boltzmann type equation, and 
 an explicit expression for  the thermal conductivity 
is obtained, which is of the form of the kinetic theory  $\kappa=\int dk C_k v_k^2 \tau_k$. Here $C_k$ is the heat capacity of phonons with wave number $k$, $v_k$ is their velocity  and $\tau_k$ is 
the average time between two collisions. 
A goal of the kinetic approach is the prediction that the mean free path $\lambda_k=v_k \tau_k$ and thus thermal conductivity are infinite  in dimension one  
when the phonon momentum is conserved.

Over the last years, several papers are devoted to achieve   
phononic Boltzmann-type
equations from microscopic dynamics (see \cite{Sp} for main ideas and tools). In
\cite{ALS}, \cite{LS}, \cite{LeS}
\cite{Pe} a kinetic
limit is performed for 
chains of an-harmonic oscillators,
 and in \cite{LSrm} a linear Boltzmann equation is
rigorously derived for the harmonic chain of oscillators with random masses.
In \cite{BOS} the authors consider a
system of harmonic oscillators in $d$ dimensions, perturbed by a weak
conservative stochastic noise. 
The following
linear Boltzmann-type equation is deduced  for the energy density
distribution, over the space $\R^d$,  of the   phonons,
characterized by a vector valued wave-number $k\in\Td$
($d$-dimensional torus)
\begin{equation}\begin{split}\label{Be0}
\partial_t u_\al(t,r,& k) +  v(k)\cdot \nabla u_\al(t,r,k)\\
&=\frac{1}{d-1}\sum_{\be\neq \al}\int_{\Td }\;dk'
R(k,k')[u_\be(t,r,k')-u_\al(t,r,k)],
\end{split}\end{equation}
$\al=1,..,d$, $d\geq 2$. Equation in dimension one is similar, 
except for the mixing of the components.  The kernel $R$ is not negative and symmetric.
Despite the exact expressions of $R$ and  $v$ (the velocity), the crucial features are that  $v$  is finite for small  $k$, i.e. $\vert v\vert\to 1$
as $\vert k\vert \to 0$, while  $R$ behaves like $\vert k\vert^2$ for small $k$, and like $\vert k'\vert^2$  for small $k'$. Na\"ively, it means  that 
phonons with small wave numbers 
travel with finite velocity, but they have low probability to be
scattered, thus  one expects that the their mean free paths 
have a macroscopic length (ballistic transport). This is  in accordance with rigorous results  showing 
that thermal conductivity is infinite in dimension one and two for a system of harmonic oscillators perturbed by a conservative noise (\cite{BOS}, \cite{BBO}). 

A  probabilistic interpretation of  (\ref{Be0}) provides an exact statement of that intuition.
The equation describes the evolution of the
probability density of a  Markov process $\left(K(t),
i(t),Y(t)\right)$ on $(\Td\times\{1,..,d\}\times \R^d)$, where
$(K(t), i(t))$ is a reversible jump process  and $Y(t)$ is a vector-valued 
additive functional of $K$, namely $Y(t)=\int_0^t ds\;
v(K_s)$.  $K$ and $i$ can be interpreted, respectively, as the  wave
number and the ``polarization" of a phonon, while $Y(t)$ denotes
its position. In order to investigate the property of 
the process 
$Y(t)$, one can look at the Markov chain
$\{X_i\}$ on $\Td$ given by the sequence of states visited by
$K(t)$, and at the waiting times $\{\tau(X_i)\}$, where $\tau(X_i)$
is the (random) time that the process spends at the $i$-th visited state. 
The vector-valued function $S_n=\sum_{i=1}^{n}\tau(X_i)v(X_i)$  gives
the value of $Y$ at the time  of the  $n$-th jump $T_n=\sum_{i=1}^n
\tau(X_i)$, then $Y(t)$ is just the piecewise
interpolation of $S_n$ at the random times $T_n$.

The behaviour of the rate $R$ implies that the stationary distribution of the chain 
is of the form $\pi(dk)\sim \vert k\vert^2dk$ for $k$
small,  and since the average of $\tau(k)$ goes like $\vert k \vert^{-2}$ for $k\ll 1$,
 the tail distribution of the random variables
$\{\tau(X_i)v(X_i)\}$ behaves like
\begin{equation}\label{tail}
\pi\left[\vert  \tau(X_i)v(X_i)\vert>\lambda\right]\sim
\frac{1}{\lambda^{1+\frac{d}{2}}}\hspace{1cm}\forall d\geq 1.
\end{equation}
Therefore,  in dimension one and two the variables $\tau(X_i)v(X_i)$ have infinite variance
with respect to the stationary measure. We remark that the variance   has  the same expression of the thermal conductivity obtained in \cite{BOS}.

The one dimensional case is discussed in \cite{BB}, where the authors 
prove that the rescaled process
$N^{-2/3}Y(N\cdot)$ converges in distribution to a symmetric L\'evy
process, stable with index $3/2$. 
Convergence of finite dimensional  marginals has been proven earlier in \cite{JKO}.
Here we consider the other critical case $d=2$. $S_n$  is now a  sum of
variables with tail distribution
$\sim\frac{1}{\lambda^2}$, which means that if they  were
independent, they would be  in the domain of attraction of a multivariate
normal distribution. Looking at the behaviour of the variance
\begin{equation*}
\pi\left[(\tau(X_i)v_\alpha(X_i))^2\mathbf{1}_{\{\vert\tau(X_i)v_\alpha(X_i)\vert\leq
\sqrt \lambda\}} \right]\sim \ln
\lambda,\hspace{1cm}\alpha\in\{1,2\},
\end{equation*}
it turns out that the proper scaling contains an extra factor $(\ln
n)^{1/2}$. The rescaled process $(n\ln n)^{-1/2} S_{nt}$ has a
central part, given by  the sum of truncated variables ${}$
$\tau(X_i)v_\alpha(X_i)\mathbf{1}_{\{\vert\tau(X_i)v_\alpha(X_i)\vert\leq
\sqrt n\}}$, with  finite variance
 and an extremal part that goes to zero in probability, due to
 the extra term $(\ln n)^{-1/2}$. This is a standard argument used
 for  sums of i.i.d. random variables with tail distribution
 (\ref{tail}), introduced for the first time by Kolmogorov and
 Gnedenko in \cite{GK}, that we  adapt  to the case of
 dependent variables.

Then we are reduced to the problem of convergence of  a sum of centered,
dependent, bounded random variables  to a  Wiener process. 
We propose two different approaches. In Section \ref{sec:proofProp},
we will use 
an abstract theorem due to  Durrett and Resnick \cite{DR}, based
on the invariance
principle for martingale difference arrays with bounded variables
(Freedman, \cite{Fr1} and \cite{Fr2}), together with a random
change of time (see, for example, Helland \cite{He} and Billingsley
\cite{Bi}). The underlying central limit theorem for  martingale
difference arrays  can be found in Dvoretzky \cite{Dv1}, \cite{Dv2}
(see also \cite{McL}, \cite{He} and references therein). 
The alternative proof, in Section \ref{sec:conv_mom},
is based on 
the convergence of 
the moments to the moments of a Brownian motion, under some asymptotic factorization conditions, and it uses  combinatorial techniques.
In this case we will only show   convergence of the finite dimensional  marginals.
%
The multidimensional generalization is based a Cram\'er-Wold argument
(see for example \cite{Bi}, \cite{Aa}, \cite{Se}, \cite{He}).
%

Convergence of $(n\ln
n)^{-1/2}S_{n\cdot}$ to a two-dimensional Wiener process 
is
in the Skorokhod $J_1$-topology. 
Moreover, since  the  random times $T_n$
are sums of positive variables with finite
expectation, one can prove, using  the  arguments in  \cite{BB},
 that 
$(n\ln n)^{-1/2}Y(n\cdot)$ converges to a two dimensional Wiener process in the uniform topology.

Finally we show that the properly rescaled solution of the linear Boltzmann equation in dimension 
two converges to diffusion. 
 The proof includes a result on the algebraic $L^2$-convergence rate of the semi-group 
(Section \ref{Sec:alg_conv}). 
 The key point is the derivation of   a  Nash 
type inequality which provides an estimate for convergence rates slower than exponential  
(\cite{Li}, \cite{BZ}, \cite{RW}).
The diffusion coefficient is given by an infrared regularization of the  
thermal conductivity obtained in \cite{BBO}, \cite{BOS}, with a proper renormalization (\ref{def:c}).

Convergence of  solutions of  linear kinetic  equations to a diffusion
under an anomalous scaling was also proved by Mellet et al \cite{MMM}, using an analytical approach. We remark 
 that they assume a  collision frequency   strictly positive, while in our case  it is zero 
in $k=0$.

The case $d\geq 3$ can be easily treated with the same strategy. In particular the rescaled solution of the Boltzmann equation converges to a diffusion equation, with a  diffusion coefficient given by the 
thermal conductivity obtained in \cite{BBO}, \cite{BOS}.

\noindent\textbf{Acknowledgements.} I wish to thank Anton Bovier and 
Nicola Kistler for 
valuable discussions.
%
Special thanks are due to the anonymous referee for 
 important remarks that essentially contributed to the final version of the manuscript. 
In particular, she/he
pointed out an incorrect step in the earlier proof of Theorem \ref{theo:conv3} and 
suggested  the argument leading to the inequality \eqref{sg} in Lemma \ref{lemma:alg_conv}.


\section{The model}
We consider equation (\ref{Be0}) in dimension two, namely
\begin{equation}\begin{split}\label{Be}
\partial_t u_\al(t,r,& k) +  v(k)\cdot \nabla u_\al(t,r,k)\\
&=\sum_{\be\neq \al}\int_{\Td }\;dk'
R(k,k')[u_\be(t,r,k')-u_\al(t,r,k)],
\end{split}\end{equation}
$\forall \alpha=1,2$, $t\geq 0$, $x\in\R^2$, $k\in\Tt$,
 with a (vector valued) velocity $v$ and a scattering kernel 
$R$ given by:
\begin{eqnarray}\label{def:v}
v_\al(k) & = &\frac{\sin(\pi k_\al)\cos(\pi
k_\al)}{\left(\sum_{\be=1}^2 \sin^2(\pi
k_\be)\right)^{1/2}},\hspace{1cm}\forall k\in\Tt,\;\forall \al\in\{1,2\}\\
\label{def:R} R(k,k') & = &16\sum_{\al =1}^2 \sin^2(\pi
k_\al)\sin^2(\pi k'_\al),\hspace{1cm}\forall k, k'\in \Tt.
\end{eqnarray}
We denote with  $(K(t), i(t))$ the jump process with values in
$\Tt\times\{1,2\}$, defined by the generator
\begin{equation}
 \label{def:generator}
\mathcal L f (\alpha, k)=\sum_{\beta\neq \alpha}\int_{\Tt}dk'\; R(k,k')\left[f(\beta, k')-f( \alpha, k) \right],
\end{equation}
with $f:\{1,2\}\times \Tt\to\R$ continuous on  $\Tt$. The process
 waits in the state $(k,i)$ an
exponential random time $\tau$ with parameter $\Phi(k,i)$
\begin{equation}\label{def:Phi}\begin{split}
\Phi(k,i)
=\sum_{j=1}^2
(1-\delta_{i,j})\int_{\Td}dk'\; R(k,k')
=8 \displaystyle\sum_{\al=1}^2\sin^2(\pi k_\al),
\end{split}\end{equation}
then it jumps to another state $(j,k')$ with probability
$
\nu\left[i,k; j,dk' \right]=(1-\delta_{i,j})\; P(k,dk'),
$
where
\begin{equation}\label{def:kerP}
P(k, dk'):=\Phi(k)^{-1}R(k,k')dk'=\frac{2\sum_\al \sin^2(\pi
k_\al)\sin^2(\pi k'_\al)}{\sum_\be \sin^2(\pi k_\be)}dk'.
\end{equation}
Observe that the two processes $K(t)$ and $i(t)$ are independent.
%
Disregarding the time, the stochastic sequence $\{X_n\}_{n\geq 0}$
of states visited by $K(t)$ is a Markov chain with value in $\Tt$,
with probability kernel $P(k, dk')$, which is 
strictly positive. Moreover,  there exists a probability measure 
$\lambda$ on $\Tt$, strictly
  positive on open sets, such that for any $k\in \Tt$ it holds
  $P(k,\cdot)\ge c_0 \lambda (\cdot)$ for some $c_0>0$.
This  implies the Doeblin condition for kernel $P$.  In view of
\cite[Thm.~16.0.2]{MT}, the discrete time Markov chain $\{X_n\}_{n\geq 0}$ is uniform ergodic. 
That is there exists a probability $\pi$ on $\Tt$ such that 
$P^n(k,\cdot)$ converges to $\pi$ in total variation uniformly with
respect to the initial condition $k$. Moreover, $\pi$ is strictly
positive on open sets.
 By direct computation
$\pi(dk)=\frac{1}{8}\Phi(k)dk$.

The process $Y(t)$, with value in $\R^2$, is an additive functional of $K(t)$ 
\begin{equation}\label{def:Y}
Y(t)=Y(0)+\int_0^t ds\; v(K_s)ds.
\end{equation}
We choose $Y(0)=0$.
In order to investigate its properties, we define two functions of the Markov chain $\{ X_n\}_{n\geq 0}$,
the clock, $T_n$, with values in $\R_+$ and the position, $S_n$, with
values in $\R^2$
\begin{equation*}
T_n  = \sum_{\ell=0}^{n-1} e_\ell \;\Phi(X_\ell)^{-1},\; \hspace{0.4cm}
S_n  = \sum_{\ell=0}^{n-1} e_\ell \;v(X_\ell)\;\Phi(X_\ell)^{-1}.
\end{equation*}
Here $\{e_\ell\}_{\ell\geq 0}$ are i.i.d. exponential random
variables with parameter $1$, and we take $S_0=0$. The clock $T_n$ is the time of the
$n$-the jump of the process $K(t)$ and   it is a sum of positive random
variables with finite expectation with respect to the invariant measure, i.e.
$\E_\pi[e_1 \;\Phi(X_1)^{-1}]=1$.  $S_n$ is a two-components vector  which gives the value
of $Y(t)$ at time $T_n$, i.e. $S_n=Y(T_n)$. It is a sum of centered random vectors whose
components show a  tail behavior  given in (\ref{tail}). Moreover,
the covariance matrix of each of these vectors is diagonal.
By denoting with  $T^{-1}$  the right-continuous inverse function of $T_n$,
i.e.
$T^{-1}(t):= \inf\{n:T_n\geq t\}$,
we can represent process $Y(t)$ as follows:
$$Y(t)=  S_{\lfloor T^{-1}(t)-1\rfloor} +
v(X_{\lfloor T^{-1}(t)-1\rfloor })(t-T_{\lfloor T^{-1}(t)-1\rfloor}),
$$
where $\lfloor \cdot\rfloor$ denotes the lower integer part.
In particular,  $Y(t)$  is the (vector valued) function defined by
linear interpolation between its values $S_{n}$ at the random points
$T_n$.

\section{Main results. }
For every $N\geq 2$, $t\geq 0$, we define the rescaled processes
\begin{eqnarray}\label{def:TN}
T_N(t) &= &\frac{1}{N}T_{\lfloor N t \rfloor},\hspace{0.4cm}
 T_N^{-1}(t)=\frac{1}{N}T^{-1}(N t),\\ \label{def:ZN}
 Z_N(t) & = & \textstyle{\frac{1}{\sqrt{N\ln N}}}S_{\lfloor N t
\rfloor}+ \left(N t - \lfloor N t \rfloor \right)
\textstyle{\frac{1}{\sqrt{N\ln N}}}v\left(X_{\lfloor N t \rfloor
-1} \right).
\end{eqnarray}
%
 Observe that $Z_N$ is a two-dimensional  continuous vector defined by linear
interpolation between its values $\frac{1}{\sqrt{N\ln N}}S_n$ at the
points $n/N$.

 We assume that the initial
distribution $\mu$ of the process $K_t$ is not concentrated in $k=0$, namely
$\forall \ve>0$ exists $\delta$ such that
\begin{equation}\label{ass:mu}
\mu\big[|k|<\delta\big]<\ve.
\end{equation}
This includes all the absolutely continuous measures w.r.t. Lebesgue measure and 
delta distributions
$\delta_{k_0}(dk)$, with $k_0\in\Tt /\{0\}$.

 Let us denote with
\begin{equation}\label{def:c}
\sigma^2:=\lim_{N\to\infty}\frac{1}{\ln N}\;\E_\pi\left[\left| \frac{e_1
 v_1 (X_1)}{\Phi(X_1)}\right|^2\mathbf{1}_{\left\{\left| \frac{e_1
 v_1 (X_1)}{\Phi(X_1)}\right| \leq \sqrt N\right\}}\right].
\end{equation}
We remark that this limit exists and  one can prove by direct computation that it is equal 
to $\frac{1}{64}\frac{1}{2\pi}$. 
By symmetry, in this
definition we can   replace $v_1(X_1)$ with $v_2(X_1)$ . We use the notation $\bar{W}_\sigma$ for the  vector valued
process  $\bar W _\sigma=(W_\sigma^1, W^2_\sigma)$, where $W_\sigma^1$ and $W_\sigma^2$ are
independent Wiener processes with marginal distribution $W^\al_\sigma(t)-W^\al_\sigma(s)\sim
\mathcal N (0, \sigma^2(t-s))$ $\forall 0\leq s <t$, $\forall \al=1,2$.

\begin{theorem}\label{theo:conv}
  Let $Z_N$ be the process defined in $(\ref{def:ZN})$. Then for
  any $0<\mathcal{T}<\infty$, $\{Z_N(t)\}_{0\leq
    t\leq \mathcal {T}}$ converges to the two-dimensional  Wiener process
   $\{\W _\sigma(t)\}_{0\leq t\leq\mathcal{T}}$.
Convergence is in distribution on the  space of continuous
  functions $C\left([0,\mathcal T], \R^2\right)$  equipped with the uniform 
topology.
\end{theorem}
Then we will prove that $\{T_N^{-1}(t)\}_{t\in [0,\mathcal T]}$
converges in distribution to the function $t$. Combining these
two results, we can show that $Z_N\circ T_N^{-1}$ converges in
distribution to $\bar{W}_\sigma$. Observing that $Z_N\circ T_N^{-1}$ is the process
$$Y_N(t)=\frac{1}{(N\ln
N)^{1/2}}\int_0^{Nt}ds\; v(K_s),$$
 this  implies our main theorem.
%
%
\begin{theorem}\label{theo:conv2}
%
For any
$0<\mathcal{T}<\infty$, $\{Y_N(t)\}_{0\leq \mathcal{T}}$ converges
to the two-dimensional Wiener process $\{\W _\sigma(t)\}_{0\leq
t\leq
    t\leq\mathcal{T}}$.
Convergence is in distribution on the  space of continuous
  functions  $C\left([0,\mathcal T], \R^2\right)$  equipped with the uniform 
topology.
\end{theorem}


Finally, we will use the previous result to show that the rescaled solution of the Boltzmann equation converges to a 
diffusion.
We denote with
$u^N$ the two dimensional vector-valued  measure defined as
$$
u^N(t,k,x):= u(Nt, k, (N\ln N)^{1/2}x), \hspace{0.4cm} \forall t\geq 0,\; \forall k\in \Tt,\; \forall x\in\R^2,$$
where $u$ is solution of (\ref{Be}) in $d=2$
with initial condition $u(0, k, x)=u_0(k,(N\ln N)^{-1/2}x )$. 
Given a function $f\in\mathcal S (\R^2\times\Tt)$- the Schwartz space, for any $a\geq 1$ 
we define the norm
$$\|f\|_{\mathcal A_a}=\left(\int_{\R^2\times\Tt}dp\,dk \big|\hat f(p,k) \big|^a \right)^{1/a},$$
where $\hat f$ is the Fourier transform of $f$ in the first variable.
We denote wit
h $\mathcal A_a$ the completion of $\mathcal S$ in the norm $\|\cdot\|_{\mathcal A_a}$.
Observe that $\mathcal A_2= L^2(\R^2\times\Tt).$
\begin{theorem}\label{theo:conv3}
Assume that $u_0\in L^2(\R^2\times\Tt;\,\R^2)\cap \mathcal A_a$, 
with   $a>2$. 
Then, $\forall t\in (0,\mathcal{T}]$, $u^N(t,\cdot,\cdot)$ converges  in 
$ L^2(\R^2\times\Tt;\,\R^2)$ -weak
to $\bar{u}(t,\cdot)$, which solves the following diffusion equation
\begin{equation}\label{eq:diff}
\begin{split}
& \partial_t \bar{u}(t,r)=\frac{1}{2}\sigma^2\; \Delta \bar{u}(t,r)\\
&\bar{u}^\alpha(0,r)=\frac{1}{2}\sum_{\beta=1,2}\int_{\Tt} dk\;u_0^\beta(r,k)
\hspace{1cm} \forall \alpha\in{1,2},\,\forall r\in\R ^2.
\end{split}
\end{equation}
\end{theorem}

\section{Sketch of the proof }
We present an outline of the proof of the main theorems. Details are postponed in Section \ref{sec:det}.
\subsection{Theorem \ref{theo:conv}}
Define the two-dimensional
random vector
\begin{equation}\label{def1}
\psi_n:=\Phi(X_n)^{-1}v(X_n), \hspace{1cm}n\in\N_0.
\end{equation}
We will denote with $\psi_n^\al$, $\al=1,2$, the $\al$-component of
$\psi_n$.

We decompose $Z_N$, defined in (\ref{def:ZN}) in two parts, i.e.
$Z_N=Z_N^{
>} +Z_N^{<}$, where $\forall t \geq 0$, $\forall \al=1,2$
\begin{eqnarray*}
Z_N^{\al >}(t)& = &(N\ln N)^{-1/2}\sum_{n=0}^{\lfloor
N t\rfloor -1}e_n \psi_n^\al \mathbf{1}_{\left\{e_n\vert
\psi_n^\al \vert
>\sqrt N\right\}} \\
 &  &+(N\ln
N)^{-1/2}e_{\scriptscriptstyle{\lfloor N t\rfloor}}
 \psi_{\scriptscriptstyle{\lfloor N t\rfloor}}^\al
 \mathbf{1}_{\{e_{\lfloor N t\rfloor}
 \vert \psi_{\lfloor N t\rfloor}^\al \vert
>\sqrt N\}}\left(N t - \lfloor N t \rfloor
\right)\\
 Z_N^{\al<}(t)& = &(N\ln
N)^{-1/2}\sum_{n=0}^{\lfloor N t\rfloor -1}e_n \psi_n^\al
\mathbf{1}_{\{e_n\vert \psi_n^\al
\vert \leq\sqrt N\}}\\
 &   &+(N\ln
N)^{-1/2}e_{\scriptscriptstyle{\lfloor Nt\rfloor}}
 \psi_{\scriptscriptstyle{\lfloor N t\rfloor}}^\al
 \mathbf{1}_{\{e_{\lfloor N t\rfloor}
 \vert \psi_{\lfloor N t\rfloor}^\al \vert
\leq \sqrt N\}}\left(N t - \lfloor N t \rfloor \right).
\end{eqnarray*}
At first  we will show that $Z_N^{>}\stackrel{P}{\to}0 $ when $N\to
\infty$.  It is enough to show that for every unitary vector
$\lambda:=(\lambda_1, \lambda_2)$
$$
\lambda_1 Z_N^{1 >}+\lambda_2 Z_N^{2
>}\stackrel{P}{\to}0,\hspace{0.4cm} N\to\infty.
$$
This is stated in the next Lemma.
\begin{lemma}\label{lemmaS+}
For every $\delta>0$
\begin{equation}
\lim_{N\to\infty}\Pro\left[\sup_{t\in[0,\mathcal{T}]}\left\vert
\lambda_1 Z_N^{1 >}(t)+\lambda_2 Z_N^{2
>}(t)\right\vert >\delta \right]=0,
\end{equation}
$\forall \lambda\in\R^2$ such that $\vert \lambda \vert =1$.
\end{lemma}
\begin{proof}
For every $\lambda\in\R^2$ with $\vert \lambda \vert =1$
\begin{equation*}\begin{split}
& \Pro\left[\sup_{t\in[0,\mathcal{T}]}\left\vert \lambda_1
  Z_N^{1 >}(t)+\lambda_2 Z_N^{2>}(t)\right\vert >\delta
  \right]\\
&\leq \Pro\left[\sup_{t\in[0,\mathcal{T}]}\left\{\left\vert
  Z_N^{1>}(t)\right\vert+ \left\vert Z_N^{2 >}(t)\right\vert\right\} >\delta
  \right]\\
&\leq \sum_{\al=1,2}
  \Pro\left[\sup_{t\in[0,\mathcal{T}]}\left\vert
  Z_N^{ \al >}(t)\right\vert>\frac{\delta}{2}
  \right]\\
\end{split}
\end{equation*}

 For every $t \in[0,\mathcal T]$, $\forall \al=1,2$
\begin{equation*}
\left\vert Z_N^{\al >}(t)\right\vert \leq
  \frac{1}{\sqrt{N\ln N}}
  \sum_{n=0}^{\lfloor N\mathcal T\rfloor -1}e_n\vert \psi_n^\al
  \vert\;
\mathbf{1}_{\left\{e_n\vert \psi_n^\al \vert >\sqrt N\right\}}.
\end{equation*}
Then, by Chebyshev's inequality
\begin{equation*}\begin{split}
&\Pro\left[\sup_{t\in[0,\mathcal{T}]}\left\vert Z_N^{\al
>}(t)\right\vert >\frac{\delta}{2} \right]\\
 & \leq 
 \frac{2}{\delta}\frac{1}{\sqrt{N\ln N}}
 \sum_{n=0}^{\lfloor N\mathcal T\rfloor -1}
 \E\left[e_n\vert \psi_n^\al \vert\;
\mathbf{1}_{\left\{e_n\vert \psi_n^\al \vert >\sqrt N\right\}}
\right]\\
&\leq \frac{2}{\delta}\frac{1}{\sqrt{\ln N}} C_0 \mathcal T,
\end{split}
\end{equation*}
where in the last inequality we used the fact that  $\forall n\geq
0$, $\forall \al=1,2$
\begin{equation*}
\E\left[e_n\vert \psi_n^\al \vert \;\mathbf{1}_{\left\{e_n\vert
\psi_n^\al \vert >\sqrt N\right\}} \right]\leq C_0\frac{1}{\sqrt N},
\end{equation*}
as one can easily compute, using the upper bound for $P^m$
(\ref{ubP}) and the fact that $\vert k\vert^{2} \vert
\psi^\al(k)\vert$ is finite for every $k\in\Tt$, $\forall \al=1,2$.

\end{proof}

Let us consider $Z_N^{<}$. As first step, we will prove that for
every unitary vector $\lambda\in\R^2$,  $\langle
Z_N^{<},\lambda\rangle:= \lambda_1 Z_N^{1 <}+\lambda_2
Z_N^{2<}\Rightarrow W_\sigma$, where $W_\sigma$ is a one dimensional Wiener
process such that $W_\sigma(t)-W_\sigma(s)\sim \mathcal N (0, \sigma^2(t-s))$. This
is stated in the following proposition, the proof is postponed  to the next section.
\begin{proposition}\label{prop:convZ_N}
Fix $\mathcal{T}>0$. Then as $N\to\infty$, for every
$\lambda\in\R^2$, with $\vert \lambda\vert=1$, $\langle
Z_N^{<},\lambda\rangle$ converges weakly to the one dimensional
Wiener process $W_\sigma$. Convergence is in distribution on the space of
continuous functions on $[0,\mathcal T]$ equipped with the uniform
topology.
\end{proposition}

Now we have to show that $Z_N^{<}$ converges to $\bar W_\sigma$.
We follow the approach of \cite{Se} (see the proof of Lemma 4). The
tightness of the sequence $\{Z_N^<\}_{N\geq 1}$ follows from the
tightness of the sequence $\{\langle Z_N^{<},\lambda\rangle\}_{N\geq
1}$, for every unitary vector $\lambda$. Thus  we only have to prove the
convergence of the finite dimensional distribution. In particular,
we have to show the following:
\begin{itemize}
\item[(i)] $Z_N^<(t)-Z_N^<(s)\Rightarrow \bar W_\sigma(t)-\bar
W_\sigma(s)$, $\forall\; 0\leq s\leq t\leq \mathcal T$;
\item[(ii)]  $Z_N^<(s)$ and
$\left(Z_N^<(t)-Z_N^<(s)\right)$ are independent, as
$N\to\infty$,\\ $\forall\; 0\leq s\leq t\leq \mathcal T$.
\end{itemize}
In order to verify the first condition, we observe that the
convergence of the process $\langle Z_N^{<}(\cdot),\lambda\rangle$
to $W_\sigma(\cdot)$ implies  that $\left(\langle
Z_N^{<}(s),\lambda\rangle, \langle
Z_N^{<}(t),\lambda\rangle \right)\Rightarrow (W_\sigma(s),
W_\sigma(t))$, for every
$s,t\geq 0$. But $(W_\sigma(s), W_\sigma(t))$ has the same
law of $\left(\langle \bar W_\sigma(s), \lambda\rangle, \langle \bar
W_\sigma(t), \lambda\rangle \right)$, then
\begin{equation*}\langle Z_N^{<}(t),\lambda\rangle-\langle
Z_N^{<}(s),\lambda\rangle\Rightarrow \langle \bar W_\sigma(t),
\lambda\rangle-\langle \bar W_\sigma(s), \lambda\rangle
\end{equation*}
 for
all $\forall\; 0\leq s\leq t\leq \mathcal T$,
$\forall \lambda \in \R^2$ with $\vert \lambda\vert=1$, and this
implies $(i)$.

In order to verify condition $(ii)$ it is sufficient to prove that
$Z_N^{<}(s)$ and $Z_N^{<}(t)-Z_N^{<}(s)$ are
asymptotically jointly Gaussian and uncorrelated. This is stated in
the next Lemma.

\begin{lemma}\label{indep}
  For all $\lambda$,
$\mu\in\R^2$
\begin{equation}
\langle Z_N^{<}(s),\lambda\rangle + \langle
(Z_N^{<}(t)-Z_N^{<}(s)), \mu\rangle   \Rightarrow \mathcal
N\left(0, \sigma^2\{\vert\lambda\vert^2s +
\vert\mu\vert^2(t-s)\}\right),
\end{equation}
$ \forall \;0\leq s< t\leq \mathcal T$.
\end{lemma}

We postpone the proof in section \ref{sec:proofind}.

\subsection{Proof of Theorem \ref{theo:conv2}}
Converge in probability  of $T_N^{-1}$  to the function $\chi$,
where $\chi(t)=t$, in a compact $[0,\mathcal T],$ is
proved as in \cite{BB}, see Lemma 8.1 and Proposition 8.2. Then 
$$(Z_N, T_N^{-1})\Rightarrow (\bar W _\sigma,
\chi)$$
(Theorem 3.9 in Billingsley \cite{Bi}) and therefore $Z_N \circ T_N^{-1}\Rightarrow \bar W _\sigma\circ
\chi$ (Billingsley \cite{Bi}, Lemma pg. 151).

\subsection{Proof of Theorem \ref{theo:conv3}}
Given a  vector valued, real function $J\in\mathcal S (\R^2; C(\Tt))$,
 we define the Fourier transform 
in the first variable
$$
\hat J(p,k)=\int_{\R^2}du\, e^{-ip\cdot u} J(u,k),\hspace{0.6cm}\forall p\in\R^2,\,k\in\Tt,
$$
and we  introduce the norm on $\mathcal S (\R^2; C(\Tt))$
\begin{equation*}
 \left\| J \right\|^2_{\mathcal B_2}=
\int_{\R^2}dp\;\Big(\sup_{k\in\Tt} \vert \hat J (p,k)\vert\Big) ^2.
\end{equation*}
We use  a probabilistic representation of the solution of the rescaled Boltzmann equation, namely
\begin{equation*}\begin{split}
& \langle J, u^N(t)\rangle\\
&=\sum_{\alpha=1,2}\int_{\R^2 \times \Tt}dp\;dk \;\hat J_\alpha(p,k)^*\E_{(\alpha,k)}\left[\hat u_0(p, \alpha_{(Nt)}, 
 K_{(Nt)})e^{-i p\cdot Y_N(t)}\right],
\end{split}
\end{equation*}
where  $\E_{(\alpha,k)}[\cdot\;]$
is the expectation starting  from the state $(\alpha,k)$,
 and
$\hat F(p, \beta, k):= \hat F_\beta (p,k)$. 
The measure $\tilde \pi$  on $\{1,2\}\times \Tt$, given by $\tilde\pi(\alpha, dk)=\frac{1}{2}dk$, is 
 invariant for the (reversible) process $\{(\alpha(t),K(t)), t\geq 0\}$ on $(\{1,2\}\times \Tt)$.

Let us choose a sequence of real numbers $\{\theta_N\}_{N\geq 1}$ such that
$\theta_N\to \infty$ for $N\uparrow \infty$ and $\frac{\theta_N}{\sqrt{N\ln N}}\to 0$. 
We show that we can replace $Y_N(t)$ with $Y_N(t-\theta_N t/N)$. Fix $R>0$. Then
\begin{equation}\label{i}\begin{split}
 &\Big|\sum_{\alpha=1,2}\int_{\R^2 \times \Tt}dp\;dk \hat J_\alpha(p,k)^*\\
&\;\;\times\E_{(\alpha,k)}\Big[\hat u_0(p, \alpha_{(Nt)}, 
 K_{(Nt)})\big(e^{-i p\cdot Y_N(t)}-e^{-i p\cdot Y_N(t-\frac{\theta_N}{N} t)}\big)\Big] \Big|\\
&\leq 
\int_{\R^2}dp\,\sup_{k\in\Tt}\big|\hat J(p,k)\big|
\mathbf 1_{\{|p|\leq R\}}\\
&\;\;\times\int_{\Tt}dk\, 
\Big|\E_{(\alpha,k)}\Big[\hat u_0(p, \alpha_{(Nt)}, 
 K_{(Nt)})\big(e^{-i p\cdot Y_N(t)}-e^{-i p\cdot Y_N(t-\frac{\theta_N}{N} t)}\big)\Big] \Big|\\
&+2\int_{\R^2}dp\,\sup_{k\in\Tt}\big|\hat J(p,k)\big|
\mathbf 1_{\{|p|> R\}}\int_{\Tt}dk\, 
\E_{(\alpha,k)}\Big[\big|\hat u_0(p, \alpha_{(Nt)}, 
 K_{(Nt)})\big|\Big].
\end{split}\end{equation}
Since 
$$
\left\vert e^{-i p\cdot Y_N(t)}-e^{-i p\cdot Y_N(t-\frac{\theta_N}{N} t)}\right\vert\leq C_0 
\frac{\theta_N}{\sqrt{N\ln N}}\vert p\vert\mathcal{T},
$$
using Cauchy-Schwartz we have that the r.h.s. of \eqref{i} is bounded by
$$
C_0 R \frac{\theta_N}{\sqrt{N\ln N}}\mathcal T\|J\|_{\mathcal B_2}\| u_0\|_{\mathcal A_2}
+C_1 \|J\|_{\mathcal B_2}
\left(\int_{\R^2\times\Tt}dp\,dk\,|\hat u_0|^2\mathbf 1_{\{|p|>R\}}\right)^{1/2}.
$$
We send $N\to\infty$ and then $R\to \infty$. 

Denoting with 
$\hat{\mathcal{U}}_p(\alpha_t, K_t)=\hat u_0(p, \alpha_{t},  K_{t})- \tilde{\pi} [\hat{u}_0](p)$, $\forall p\in\R^2$, $\forall t>0$, we have
\begin{equation*}
 \begin{split}
\E_{(\alpha,k)}\left[\big(\hat u_0(p, \alpha_{(Nt)}, 
 K_{(Nt)})- 
 \tilde{\pi} [\hat{u}_0](p)\big) 
 e^{-i p\cdot Y_N(t-\frac{\theta_N}{N} t)}\right]\\
=  \E_{(\alpha,k)}\left[e^{-i p\cdot Y_N(t-\frac{\theta_N}{N} t)}\; S_{\theta_N t}\;\hat{\mathcal{U}}_p(\alpha_{t-\theta_N t}, K_{t-\theta_N t})\right],
 \end{split}
\end{equation*}
where $\{S_t\}_{t\geq 0}$ is the semigroup associated to 
the generator (\ref{def:generator}).Thus, using Cauchy-Schwartz,
\begin{equation}\label{ii}\begin{split}
 &\Big|\sum_{\alpha=1,2}\int_{\R^2 \times \Tt}dp\;dk \hat J_\alpha(p,k)^*\\
&\;\;\times\
\E_{(\alpha,k)}\left[\big(\hat u_0(p, \alpha_{(Nt)}, 
 K_{(Nt)})- 
\tilde{\pi} [\hat{u}_0](p)\big)
 e^{-i p\cdot Y_N(t-\frac{\theta_N}{N} t)}\right]\Big\vert\\
&\leq 2\;\left\|J\right\|_{\mathcal A_2}\; \left(\int_{\R^2}dp\; {\| S_{\theta_N t}\;\hat{\mathcal{U}}_p\|}^2_{L_{\tilde{\pi}}^2}\right)^{1/2}.
\end{split}
\end{equation}
In order to prove that the last expression converges to zero, we use the following lemma on the $L^2$-convergence.
\begin{lemma}\label{lemma:alg_conv}
For every $f\in L^2_{\tilde\pi}$ with $\tilde\pi\big[ f]=0$ 
the following inequality holds:
\begin{equation}\label{alg_conv}
\|S_t f\|^2_{L_{\tilde{\pi}}^2}\leq C \| f\|^2_{L^{q}_{\tilde{\pi}}}\frac{1}{t^{1-\frac 2 q}}\, ,
\hspace{0.4cm}q>2,
\end{equation}
for every $t\geq 0$.
\end{lemma}
We postpone the proof in Section \ref{Sec:alg_conv}.
Then 
$$
\int_{\R^2}dp\; {\| S_{\theta_N t}\;\hat{\mathcal{U}}_p\|}^2_{L_{\tilde{\pi}}^2}
\leq C\frac{1}{(\theta_Nt)^{1-\frac{2}{q}}}\int_{\R^2}dp\;\| \hat{\mathcal{U}}_p\|^2_{L^{q}_{\tilde{\pi}}}
$$
and
the r.h.s. of \eqref{ii} is bounded by
$$
 C_1 \|J\|_{\mathcal A_2}\|u_0 \|_{\mathcal A_q}\frac 1 {(\theta_N t)^{\frac{q-2}{2q}}}, \;\; q>2,
$$
which converges to zero for $N\to\infty$.
Finally,    
we can replace 
$\E_{(\alpha, k)}\big[e^{-ip Y_N(t)  } \big]$
with 
$\exp \{-\frac 1 2|p|^2\sigma^2 t\}$
. We have
\begin{equation}\label{iii}\begin{split}
&\left|\sum_{\alpha}\int_{\R^2\times\Tt}dp\,dk\, \hat J_\alpha(p,k) \tilde \pi [\hat u_0(p)]
\E_{(\alpha, k)}\big[e^{-ip Y_N(t)}- e^{-\frac 1 2|p|^2\sigma^2 t}\big]
\right|\\
&\leq C_0 \| J \|_{\mathcal B_2} \left(\int_{\R^2} dp\, \big|\tilde \pi [\hat u_0(p)]\big|^2
\mathbf 1_{\{|p|\geq R\}}\right)^{1/2}\\
&+ 
\int_{\R^2} dp\,\sup_{k\in\Tt}\big|\hat J(p,k)\big| \big|\tilde \pi [\hat u_0(p)]\big|
\mathbf 1_{\{|p|\leq R\}}\\
&\;\; \times
\int_{\Tt}dk\, \Big| \E_{(\alpha, k)}\big[e^{-ip Y_N(t)}- e^{-\frac 1 2|p|^2\sigma^2 t}\big]  \Big|,
\end{split}\end{equation} 
for any $R>0$. By Theorem \ref{theo:conv2}, the second integral on the r.h.s. converges to zero 
for $N\to\infty$, $\forall t\in[0, \mathcal T]$, then we send $R\to\infty$.

We conclude the proof by observing that, since
$$
\|S_t u^N(t)\|^2_{L^2(\R^2\times\Tt)}\leq \| u_0\|^2_{L^2(\R^2\times\Tt)},\qquad\forall N\geq1, \forall t\geq 0,
$$
then there   exists $\tilde u(t)\in L^2(\R^2\times\Tt)$ such that $u^N(t)$ weakly converges to $\tilde u(t)$
as $N\to\infty$. Moreover, we have just proved that for every $J\in\mathcal S$
$\langle J, u^N(t)\rangle \to \langle J, \bar u(t)\rangle$ as $N\to\infty$, for any $t>0$, where
$\bar u(t)$ is solution of \eqref{eq:diff}. Therefore, using the fact that the Schwartz space $\mathcal S$ is dense in $L^2$, we have $u^N(t)\to \bar u(t)$ weakly in $L^2(\R^2\times\Tt)$.

\subsection{Algebraic convergence rate}\label{Sec:alg_conv}
Suppose that,  for every $f\in L^2_{\tilde\pi}$
such that $\tilde \pi \big[f\big]=0$, the following weak Poincar\'e inequality holds: 
\begin{equation}\label{wPin}
\|f\|^2_{L^2_{\tilde\pi}} \leq \frac{C_0}
{r^{a-1}}
\mathcal E (f,f)
+r \|f \|^2_{L^q_{\tilde\pi}},\hspace{0.4cm}a>1,\;\;q>2,\;\forall r>0,
\end{equation}
where $\mathcal E (f,f)$ is the Dirichelet form.
By optimizing on $r$, one gets the following
Nash type inequality: 
\begin{equation*}\label{Nash}
\|f\|^2_{L^2_{\tilde\pi}} \leq C\big[\mathcal E(f,f) \big]^{\frac 1 a}
\Big(\|f \|^2_{L^q_{\tilde\pi}}\Big)^{1-\frac 1 a},\hspace{0.4cm}\,\; q>2,\; a>1.
\end{equation*}
The $L^q_{\tilde\pi}$ norm is defined in a dense subset of $L^2_{\tilde\pi}$. Moreover, the $L^q_{\tilde\pi}$ norm is  monotone under the semi-group $\{S_t\}_{t\geq 0}$,  namely $\|S_t f \|^2_{L^q_{\tilde\pi}}\leq \|f \|^2_{L^q_{\tilde\pi}}$ $\forall t\geq 0$, for every $q\geq 1$ (contractivity property of a Markov semi-group).
Therefore,  we can apply Theorem 2.2 of \cite{Li} (see also  \cite{RW} and \cite{BZ}) and we get
 the following algebraic rate of convergence 
\begin{equation*}
\|S_t f\|^2_{L_{\tilde{\pi}}^2}\leq C \| f\|^2_{L^{q}_{\tilde{\pi}}}\frac 1 
{t^{1/(a-1)}},
\hspace{0.4cm}q>2,
\end{equation*}
which holds for every  $f\in L^2_{\tilde\pi}$.
Then, in order to prove  Lemma \ref{lemma:alg_conv}, it suffices  to show that \eqref{wPin} holds.

The Dirichelet form has the following expression:
\begin{equation*}\begin{split}
\mathcal E (f,f) 
 &=\frac 1 2 \sum_{\alpha=1,2}\sum_{\beta\neq\alpha}\int_{\Tt}dk\,f(\alpha,k)
\int_{\Tt}dk'\,R(k,k')\big[ f(\beta,k')-f(\alpha, k) \big]\\
& =\frac 1 2 \sum_{\alpha=1,2}\int_{\Tt\times\Tt} dk\, \Phi(k)f(\alpha,k)
\left[1- P \right]f(\alpha,k),\\
\end{split}\end{equation*}
where $P$ is the operator acting 
the vector-valued functions $f:\Tt\to\R^2$
$$
P f(\alpha,k)=\sum_{\beta\neq\alpha}\int_{\Tt}P(k,dk')f(\beta,k'),\hspace{0.4cm}\forall\alpha=1,2.
$$
Here $P(k,dk')$ is the probability kernel defined in \eqref{def:kerP}. The corresponding invariant measure is  $\pi(\alpha,dk)=\frac 1 {16}\Phi(k)dk$. 
Since the operator $ P$ is compact with a positive kernel $P(k,dk')$, using the same arguments of  \cite{JKO}, Lemma 3.2, one can show  that $0$ is a simple eigenvalue for $1- P$, and therefore 
the following gap estimate is obtained 
\begin{equation}\label{sg}
\mathcal{E}(f,f)\geq c\sum_{\alpha=1,2}\int_{\Tt}dk\,\Phi(k)|f(\alpha,k)-\pi[f]|^2,
\end{equation}
with $c>0$ and $\pi[f]$   the expectation value with respect to the measure 
$\pi(\alpha, dk)$.
We   define the set 
$A_\delta=\{k\in\Tt: |k|>\delta  \}$, with $\delta\in (0,1)$, 
and we denote by $A_\delta^c$ its complement. 
Then the r.h.s.  of \eqref{sg}    
 is bounded from below by
\begin{equation*}
\begin{split}
&c\sum_{\alpha=1,2}\int_{\Tt}dk\,\Phi(k)\mathbf 1_{\{A_{\delta}\}}|f(\alpha,k)-\pi[f]|^2\\
&\geq c_1\inf_{\{k\in A_\delta\}}\Phi(k)\sum_{\alpha=1,2}\int_{A_\delta}dk|f(\alpha,k)-\pi[f]|^2.
\end{split}
\end{equation*}
We observe that 
\begin{equation*}\begin{split}
\sum_{\alpha=1,2}\int_{A_\delta}dk|f(\alpha,k)-\pi[f]|^2
 \geq &\|f\mathbf 1_{\{A_{\delta}\}}\|^2_{L^2_{\tilde\pi}}-2\pi[f]\,\tilde\pi[f\mathbf 1_{\{A_{\delta}\}}]\\
=&\|f\mathbf 1_{\{A_{\delta}\}}\|^2_{L^2_{\tilde\pi}}+2\pi[f]\,\tilde\pi[f\mathbf 1_{\{A_{\delta}^c\}}]
\end{split}\end{equation*}
where in the last equality we use the fact that $\tilde\pi[f]=0$. Since
$\inf_{\{k\in A_\delta\}}\Phi(k)=c_1\delta^2$, 
we obtain
\begin{equation}\label{eff0}\begin{split}
\big\|f\mathbf 1_{\{A_{\delta}\}}\|^2_{L^2_{\tilde\pi}}
&\leq \frac {C}{\delta^2}\mathcal E (f,f)-2\pi[f]\,\tilde\pi[f\mathbf 1_{\{A_{\delta}^c\}}]\\
&\leq \frac {C}{\delta^2}\mathcal E (f,f)+C'\|f\|^2_{L^p_{\tilde\pi}}\,
\Big(\tilde\pi[A^c_{\delta}]\Big)^{1-\frac 1 p},
\end{split}\end{equation}
with $p>1$.
Now we observe that 
\begin{equation*}\begin{split}
\big\|f\|^2_{L^2_{\tilde\pi}}&=\big\|f\mathbf 1_{\{A_{\delta}\}}\|^2_{L^2_{\tilde\pi}}
+\big\|f\mathbf 1_{\{A^c_{\delta}\}}\|^2_{L^2_{\tilde\pi}}\\
&\leq \big\|f\mathbf 1_{\{A_{\delta}\}}\|^2_{L^2_{\tilde\pi}}+
 \| f \|^2_{L^{2b}_{\tilde\pi}}\,
\Big(\tilde\pi[A^c_{\delta}]\Big)^{1-\frac1 b},\;\;b>1.
\end{split}\end{equation*}
and since $\tilde\pi[A_\delta^c]=\delta^2$, finally we get 
\begin{equation*}
\big\|f\|^2_{L^2_{\tilde\pi}}\leq \frac {C} {\delta^2} \mathcal E (f,f)+C'(\delta^2)^{1-\frac1 b}
\| f \|^2_{L^{2b}_{\tilde\pi}},\;\; b>1.
\end{equation*}
Setting $r=C'\delta^{2(1-\frac 1 b)}$ and $q=2b$, we get the weak Poincar\'e inequality \eqref{wPin}
with $a-1=\frac q {q-2}$.
\subsection*{Remark.} We can extend this proof to the general case of the process in $d$-dimensions. 
We get the following algebraic convergence rate:
\begin{equation}\label{alg_conv_d}
\|S_t f\|^2_{L_{\tilde{\pi}}^2}\leq C \| f\|^2_{L^{p}_{\tilde{\pi}}}\frac{1}{t^{\frac d 2 (1-\frac 2 q)}},
\;\;q>2,\; \forall d\geq 1.
\end{equation}

\section{Details}\label{sec:det}
We start with some preliminary results on $P^m$, the $m-$th
convolution integral of $P$, the probability kernel defined in
\ref{def:kerP}.
By direct computation
\begin{equation}\label{form:prob}
P^m(k,dk')=\frac{2}{\sum_{\gamma=1}^2\sin^2(\pi
k_\gamma)}\sum_{\al=1}^2\sum_{\beta=1}^2\sin^2(\pi
k_\al)A^{(m)}_{\al,\beta}\sin^2(\pi k'_\beta)dk'
\end{equation}
where, $\forall \al,\beta\in\{1,2\}$,
\begin{equation}
A^{(1)}_{\alpha,\beta}=\delta_{\alpha,\beta},\hspace{1cm}A^{(m+1)}_{\al,\beta}=\left[a^{m}\right]_{\al,\beta}
\hspace{0.4cm}\forall m\geq 1.
\end{equation}
Here $a$ is a $2\times 2$ real matrix with elements
\begin{eqnarray*}
a_{11}=a_{22} &=& 2\int_{\Tt}dk\; \frac{\sin^4(\pi k_1)}{\sum_{\alpha}\sin^2(\pi k_\alpha)},\\
a_{12}=a_{21} &=& 2\int_{\Tt}dk\; \frac{\sin^2(\pi k_1)\sin^2(\pi k_2)}{\sum_{\alpha}\sin^2(\pi k_\alpha)}.
\end{eqnarray*}
Observe that the condition
\begin{equation*}
\int_{\Tt} P^m(k,dk')=1\hspace{1cm}\forall m\geq 1,
\end{equation*}
implies
\begin{equation}\label{bound:A}
\sum_{\beta=1}^2 A_{\al,\beta}^{(m)}=1,\hspace{1cm}\forall
\al=1,2,\;\; \forall m\geq 1,
\end{equation}
and thus
\begin{equation}\label{ubP}
P^m(k,dk')\leq 2\sum_{\beta=1,2}\sin^2(\pi k'_\beta)dk',
\hspace{0.6cm}\forall k\in\Tt,\; \forall m\geq 1.
\end{equation}

\subsection{Proof of Proposition
\ref{prop:convZ_N}}\label{sec:proofProp} Fix
$\lambda:=(\lambda_1,\lambda_2)$ with $\lambda_1^2+\lambda_2^2=1$.
We will follow the strategy of Durrett and Resnick \cite{DR} to
prove that $\langle Z^<_N,\lambda\rangle:=\lambda_1
Z_N^{1<}+\lambda_2 Z_N^{2 <}$ converges weakly to a Wiener process
$W_c$. They use a result of Freedman \cite{Fr1}, pages 89-93, on
martingale difference arrays with uniformly bounded variables. We
start with the following

\begin{definition}\label{def:mda}
A collection of random variables $\{\xi_{N,i}\}$, $N\geq 1$, $i\geq
1$ and $\sigma$-fields $\mathcal{F}_{N,i}$, $i\geq 0$, $N\geq 1$ is
a martingale difference array if
\begin{itemize}
\item[(i)] for all $N\geq 1$, $\mathcal{F}_{N,i}$, $i\geq 0$ is a
nondecreasing sequence of $\sigma$-fields;
\item[(ii)] for all $N\geq 1$, $i\geq 1$, $\xi_{N,i}$ is
$\mathcal{F}_{N,i}$ measurable;
\item[(iii)] for all $N\geq 1$, $E\left[\xi_{N,i}\vert
\mathcal{F}_{N,i-1}\right]=0$ a.s.
\end{itemize}
\end{definition}

We introduce the following notations:
\begin{equation}\begin{split}\label{def:lambdapsi}
\langle \lambda, \bar \Psi_{N,m}\rangle := & \lambda_1
\frac{e_m\psi_m^1}{\textstyle{\sqrt{N\ln
N}}}\mathbf{1}_{\left\{e_m\vert
   \psi_m^1 \vert \leq \sqrt N\right\}}\\
  & + \lambda_2 \frac{ e_m \psi_m^2}{\textstyle\sqrt{N\ln N}} \mathbf{1}_{\left\{e_m\vert
    \psi_m^2 \vert \leq \sqrt N\right\}},
\end{split}\end{equation}
$\forall N\geq 2, m\geq 0$, and, for $N=1$, $m\geq 0$ 
\begin{equation*}
\langle \lambda, \bar \Psi_{1,m}\rangle =
     \lambda_1 e_m\psi_m^1 \mathbf{1}_{\left\{e_m\vert
   \psi_m^1 \vert \leq 1\right\}} + \lambda_2 e_m\psi_m^2 \mathbf{1}_{\left\{e_m\vert
    \psi_m^2 \vert \leq  1\right\}}.
\end{equation*}

For all $N\geq 1$, $m\geq 0$, we denote with ${\mathcal
F}_{N,\;m}$ the $\sigma$-field generated by
$\{X_0,..,X_m\}\times\{e_0,..,e_m\}$, where $\{X_m\}_{m\geq 0}$ is
the Markov chain with value in $\Tt$.
Then we observe that $\{\langle \lambda,\bar \Psi_{N,m}\rangle,
\mathcal{F}_{N,\;m}\}_{ N\geq 1, m\geq 1}$ is a martingale
difference array. In particular, condition (iii) of $\ref{def:mda}$
can be easily checked using the explicit form of probability kernel
$P[k,dk']$ .

By definition, the variables $\langle \lambda,\bar
\Psi_{N,m}\rangle$ are
 uniformly bounded in $m$, i.e. for all $N\geq 1$
$\left\vert\langle \lambda,\bar \Psi_{N,m}\rangle\right\vert \leq
\ve_N$ , $\forall m\geq 0$,
 where $\ve_N=\frac{2}{\sqrt{\ln N}}$ if $N\geq 2$,  and $\ve_1=2$.
 In particular $\ve_N \downarrow 0$ when $N\to\infty$.

For every $N\geq 1$, $j\geq 1$, let us define
\begin{eqnarray}\label{def:S}
 \langle \lambda,S_{N,j}\rangle& = &
  \sum_{m=1}^{j} \langle\lambda,
\bar \Psi_{N,m}\rangle,\\\label{def:V} \langle\lambda,
V_{N,j}\rangle& = &\sum_{m=1}^{j}E\left[\langle \lambda,\bar
\Psi_{N,m}\rangle^2\vert \mathcal{F}_{N,\;m-1}\right].
\end{eqnarray}
We will prove  in lemma \ref{lemmaV} that
$\Pro\left[\lim_{j\to\infty}\langle\lambda,V_{N,j}\rangle=\infty\right]=1$,
for all $N\geq 1$, i.e. the martingale difference array
$\{\langle\lambda, \bar \Psi_{N,m}\rangle, \mathcal{F}_{N,\;m}\}_{
N\geq 1, m\geq 0}$  satisfies the hypotheses of Theorem 2.1 in
\cite{DR}. Thus, setting
\begin{equation*}
j_{N,\lambda}(t)=\sup\{j\vert \langle\lambda,V_{N,j}\rangle\leq
t\},
\end{equation*}
 we get that
$\langle\lambda,S_{N,j_{N,\lambda}(\cdot)}\rangle$ converges weakly
as a sequence of random elements of $D[0,\mathcal T]$ to a standard
Wiener process $W$.



Now let $\phi_{N,\lambda}(t)=\langle\lambda,V_{N,\lfloor
N\theta\rfloor}\rangle$, $\forall t\in[0,\mathcal T]$. By
definition
 $$j_{N,\lambda}\circ \phi_{N,\lambda}(t)=\lfloor
N t\rfloor.$$
 In order to prove that $\phi_{N,\lambda}$
converges in probability to the function $\phi:$
$\phi(t)=\sigma^2 t$, it suffices to show that
$\phi_{N,\lambda}(t)\stackrel{P}{\to} \sigma^2 t$, $\forall
t\in[0,\mathcal T]$, since $\phi$ is continuous and
$\phi_{N,\lambda}$ is monotone. That  will be proved in lemma
\ref{lemmaV}.
%
Then
$$(\langle\lambda,S_{N,j_{N,\lambda}}\rangle,
\phi_{N,\lambda})\Rightarrow (W, \phi),$$
(Billingsley \cite{Bi}, Theorem 3.9)
 and therefore
$$\langle\lambda,S_{N,j_{N,\lambda}}\rangle\circ\phi_{N,\lambda}\Rightarrow
W\circ\phi$$ 
(Billingsley \cite{Bi}, Lemma pg. 151).

Finally, $$\langle\lambda,S_{N,\lfloor
N\cdot\rfloor}\rangle=\langle\lambda,S_{N,j_N(\phi_N(\cdot))}\rangle\Rightarrow
W_\sigma^2,$$ where convergence is in distribution
 on the space $D[0,\mathcal T]$ equipped with the Skorokhod
 $J_1$-topology.

The process $\langle\lambda,{\tilde S}_N(t)\rangle:=
\sum_{m=0}^{\lfloor N t\rfloor -1} \langle \lambda,\bar
\Psi_{N,m}\rangle$ 
converges  also to $W_\sigma$. For every $N\geq 2$, $\langle
Z_N^{<},\lambda\rangle=\lambda_1 Z_N^{1 <} + \lambda_2 Z_N^{2<}$ is
the continuous function defined by linear interpolation between its
values $\langle\lambda, \tilde{S}_{N}(m/N)\rangle$ at points $m/N$.
The two sequences
$\{ \langle\lambda,\tilde{S}_{N}(t)\rangle, 0\leq t\leq\mathcal T \}$ and 
$\{\langle Z_N(\theta),\lambda\rangle 0\leq t\leq\mathcal T\}$ are
asymptotically equivalent, i.e. if either converges in distribution 
as $N\to\infty$, then so does the other. Convergence of 
$\langle Z_N^{<},\lambda\rangle$ to $W_\sigma$ is 
in distribution on the space of continuous functions equipped with the uniform topology.

We conclude this subsection with the main Lemma.
\begin{lemma}\label{lemmaV}
For every $N\geq 1$, for every unitary vector $\lambda\in\R^2$,
\begin{equation}
\Pro\left[\lim_{j\to\infty}\langle\lambda,V_{N,j}\rangle=\infty\right]=1.
\end{equation}

Moreover, for every $\delta>0$, for every unitary vector
$\lambda\in\R^2$,
\begin{equation}\label{eq:V1}
\lim_{N\to\infty} \Pro\Big[\left\vert\langle\lambda, V_{N,\lfloor
N\theta\rfloor}\rangle-\sigma^2\theta\right\vert
>\delta\Big]=0,
\end{equation}
$\forall \theta\in[0,\mathcal T]$.
\end{lemma}

\begin{proof}
Fix $\lambda\in\R^2$, with $\vert\lambda\vert^2=1$. $\forall N\geq
2$, we define $f_N:\Tt\to\R^2$
\begin{equation}\label{lm0}\begin{split}
f_N(k) =&\int_0^\infty dz\; e^{-z}\\
&\times \int_{\Tt}P(k,dk')\left(
\sum_{\al=1,2}\lambda_\al\frac{ z\; \psi^\al(k')}{\sqrt{N\ln
N}}\mathbf{1}_{\{z\;\vert\psi^\al(k') \vert \leq \sqrt
N\}}\right)^2.
\end{split}\end{equation}
Using (\ref{form:prob}), 
we get 
$f_N(k)\geq C_0/N$, with $0\leq C_0<\infty.$
Since 
 $$f_{N}(X_m)=\E\left[\langle \bar \Psi_{N,m+1},\lambda\rangle^2 \vert
\mathcal{F}_m\right],\hspace{1cm}\forall m\geq 0$$
 then,  for all $N\geq 1$, $\langle\lambda,V_{N,j}\rangle\geq  j\;C_0 N^{-1}$ which goes to infinity 
 for $j\to\infty$, a.s.

Now we focus on (\ref{eq:V1}). By Chebychev inequality, for every
$N\geq 1$
\begin{equation}\label{lm1}\begin{split}
&\Pro\Big[\big\vert \langle\lambda,V_{N,\lfloor
Nt\rfloor}\rangle-\sigma^2t\big\vert
>\delta\Big]\\
&
   \leq \Pro\left[\Big\vert
      \sum_{n=1}^{\lfloor Nt\rfloor}\Big(\E\left[\langle \bar \Psi_{N,n},\lambda\rangle ^2\vert
{\mathcal F}_{n-1}\right] -\frac{\sigma^2}{N}\Big)
     \Big\vert > \delta-\frac{1}{N}\right]\\
&\leq \frac{1}{{\tilde \delta}_N ^2}\sum_{n=1}^{\lfloor
Nt\rfloor}
          \E\left[ \Big(\E\left[\langle \bar \Psi_{N,n},\lambda\rangle ^2\vert
{\mathcal F}_{n-1}\right] -\frac{\sigma^2}{N}\Big)^2  \right]\\
&\quad +\frac{1}{{\tilde \delta}_N ^2}\sum_{n=1}^{\lfloor
         Nt\rfloor}\sum_{m\neq n}
\E\left[  \Big(\E\left[\langle \bar \Psi_{N,n},\lambda\rangle ^2\vert
{\mathcal F}_{n-1}\right] -\frac{\sigma^2}{N}\Big)\right.\\
&\hspace{1.5cm}\left.\times
\Big(\E\left[\langle \bar
\Psi_{N,m},\lambda\rangle ^2\vert {\mathcal F}_{m-1}\right]
-\frac{\sigma^2}{N}\Big)\right],
\end{split}\end{equation}
where ${\tilde \delta}_N=\delta -N^{-1}$. By (\ref{ubP}), we get
\begin{equation}\label{ubf}
 \E\left[\langle \bar
\Psi_{N,m},\lambda\rangle ^2\vert {\mathcal F}_{m-1}\right]=f_N(X_{m-1})\leq \frac{C_0}{N},
\end{equation}
thus the first sum on the r.h.s. of (\ref{lm1}) is bounded by ${\tilde \delta}_N ^{-2}C_1 T/N$,
with $C_1$ finite.
Let us consider the second sum on the r.h.s. of (\ref{lm1}).
%
For $n>m$
\begin{equation*}\begin{split}
&\E\Big[\E\left[\langle \bar \Psi_{N,n},\lambda\rangle
 ^2\vert {\mathcal F}_{n-1}\right]
 \E\left[\langle \bar \Psi_{N,m},\lambda\rangle
 ^2\vert {\mathcal F}_{m-1} \right]\Big]\\
 & = \E\Big[ \E\left[\langle \bar \Psi_{N,m},\lambda\rangle
 ^2\vert {\mathcal F}_{m-1} \right]
  \E\big[ \E\left[\langle \bar \Psi_{N,n},\lambda\rangle
 ^2\vert {\mathcal F}_{n-1}\right]  \vert{\mathcal F}_{m-1} \big]
 \Big].
\end{split}\end{equation*}
We set
$$g_N^{n-m}(X_{m-1}):=
 \E\Big[\E\left[\langle \bar \Psi_{N,n},\lambda\rangle
 ^2\vert {\mathcal F}_{n-1}\right]\vert {\mathcal F}_{m-1}\Big],
$$
where, for every $l\geq 1$, $N\geq 1$, the function $g:\Tt\to\R^2$ is given by
$$
g_N^l(k)= \int_{\Tt} dk'\; P^l(k,dk')f_N(k'),
$$
with $f_N$ defined in (\ref{lm0}). By (\ref{ubP}) and (\ref{ubf}) we get
\begin{equation}\label{lm4}
g_N^l(k)\leq \frac{C_0}{N},\hspace{1cm}\forall k\in\Tt,\;\forall l\geq 1.
\end{equation}
We fix $M$, $1\leq M<N$ and we get
\begin{equation*}
\begin{split}
&\sum_{n=1}^{\lfloor
         Nt\rfloor}\sum_{m\neq n}\E\Big[
 \E\left[\langle \bar \Psi_{N,n},\lambda\rangle
 ^2\vert {\mathcal F}_{n-1}\right]
 \E\left[\langle \bar \Psi_{N,m},\lambda\rangle
 ^2\vert {\mathcal F}_{m-1}\right]\Big]\\
 &=2\sum_{m=1}^{M}\sum_{n=m+1}^{\lfloor Nt\rfloor}
\E\Big[f_N(X_{m-1})g_N^{n-m}(X_{m-1})
 \Big]\\
& \;\;+2\sum_{m=M+1}^{\lfloor
Nt\rfloor}\sum_{n=m+1}^{m+M}\E\Big[f_N(X_{m-1})g_N^{n-m}(X_{m-1})\Big]
\\&
 \;\;+2\sum_{m=M+1}^{\lfloor Nt\rfloor}\sum_{n=m+M+1}^{\lfloor
Nt\rfloor}\E\Big[f_N(X_{m-1})g_N^{n-m}(X_{m-1})\Big].
\end{split}\end{equation*}
By (\ref{lm4}), the first and the second sum on the r.h.s. are
bounded form above by $C \mathcal T M/N$, with $C$ finite. We denote
by $\mu P^{m-1}$  the convolution integral of the initial measure
$\mu$ and the probability $P^{m-1}$. For every $l\geq 1$,
\begin{equation*}
\begin{split}
\E\Big[f_N(X_{m-1})g_N^{l}(X_{m-1})\Big]= &
\E_{\pi}\Big[f_N(X_{m-1})g_N^{l}(X_{m-1})\Big]\\
&+ \int_{\Tt} \left[\mu P^{m-1}(dk)-\pi(dk)\right]f_N(k)g_N^l(k)\\
\end{split}
\end{equation*}
where the last term is bounded by $C' N^{-2}\int_{\Tt} \left\vert\mu
P^{m-1}(dk)-\pi(dk)\right\vert$.
 Moreover, for every $l\geq 1$
\begin{equation*}
\begin{split}
&\E_{\pi}\Big[f_N(X_{m-1})g_N^{l}(X_{m-1})\Big]\\ &=  \int_{\Tt}
\pi(dk)f_N(k)\int_{\Tt} dk' P^l(k,dk')f_N(k')\\
& \leq \left(\int_{\Tt} \pi(dk)f_N(k)\right)^2 +\frac{C'
}{N^2}\int_{\Tt} \left\vert\mu P^{m-1}(dk)-\pi(dk)\right\vert.
\end{split}
\end{equation*}
We get
\begin{equation*}
\begin{split}
&\sum_{n=1}^{\lfloor
         Nt\rfloor}\sum_{m\neq n}\E\Big[
 \E\left[\langle \bar \Psi_{N,n},\lambda\rangle
 ^2\vert {\mathcal F}_{n-1}\right] \E\left[\langle \bar
 \Psi_{N,m},\lambda\rangle
 ^2\vert {\mathcal F}_{m-1}\right]\Big]\\
& \leq \lfloor Nt\rfloor (\lfloor Nt\rfloor-1)
\left(\E_\pi\left[\langle \bar \Psi_{N,1},\lambda\rangle
 ^2 \right]\right)^2\\
 & \quad + C \mathcal T \frac{M}{N}+ C' \mathcal T \int_{\Tt} \left\vert\mu
 P^{M}(dk)-\pi(dk)\right\vert,
\end{split}\end{equation*}
we $C$ and $C'$ finite. In the same way one can prove that
\begin{equation*}
\begin{split}
\sum_{n=1}^{\lfloor
         N t\rfloor}\E\left[\langle \bar \Psi_{N,n},\lambda\rangle
 ^2 \right] & \leq \lfloor Nt\rfloor\E_\pi\left[\langle \bar \Psi_{N,n},\lambda\rangle
 ^2 \right]\\
 & \quad+ C \mathcal T \frac{M}{N}
 + C' \mathcal T \int_{\Tt} \left\vert\mu
 P^{M}(dk)-\pi(dk)\right\vert,
\end{split}\end{equation*}
with some $C$, $C'$ finite, and finally we get
\begin{equation*}\begin{split}
\Pro\Big[\big\vert \langle\lambda,V_{N,\lfloor
Nt\rfloor}\rangle-\sigma^2t\big\vert
>\delta\Big]
\leq &\frac{1}{{\tilde \delta}^2_N}\; C \mathcal T \frac{M}{N}\\
&+
\frac{1}{{\tilde \delta}^2_N}\; C' \mathcal T \int_{\Tt}
\left\vert\mu
 P^{M}(dk)-\pi(dk)\right\vert,
\end{split}\end{equation*}
where $C$, $C'$ are finite. (\ref{eq:V1}) is proved by sending $M,
N\to\infty$ in such a way that $M/N\to 0$.

\end{proof}

\subsection{Proof of Lemma (\ref{indep})}\label{sec:proofind}
We use the central limit theorem for martingale difference array
(\cite{Dv1}, Theorem 1; see also \cite{Dv2}, \cite{He}) which states
the follows: fix $t>0$, and let
$\{\xi_{N,i},\mathcal{F}_{N,i}\}_{N\geq 1,\;i\geq 0}$  be a
martingale difference array such that
\begin{eqnarray*}
&(i)& \sum_{i=1}^{\lfloor N t\rfloor} \E\left[\xi_{N,i}^2\vert
{\mathcal F}_{N, i-1} \right]\stackrel{P}{\to}c t, \hspace{1cm}
N\uparrow\infty;\\
&(ii)&\sum_{i=1}^{\lfloor N t\rfloor} \E\left[\xi_{N,i}^2
\mathbf{1}_{\{\vert\xi_{N,i} \vert>\ve\}}\vert {\mathcal F}_{N, i-1}
\right]\stackrel{P}{\to}0, \hspace{1cm} N\uparrow\infty,\quad
\forall \ve>0.
\end{eqnarray*}
Then $$\sum_{i=1}^{\lfloor N t\rfloor} \xi_{N,i}\Rightarrow
{\mathcal N}(0, c t).$$

By definition of $Z_N^<$, $\forall \lambda\in\R^2$
\begin{equation}\label{hi}
\langle\lambda,Z_N^<(t)\rangle=\langle\lambda, S_{N,\lfloor
Nt\rfloor}\rangle+(Nt-\lfloor Nt\rfloor)\langle\lambda,
{\bar\Psi}_{\lfloor Nt\rfloor}\rangle,
\end{equation}
$\forall t\in[0, \mathcal T]$, where $\langle\lambda,
S_{N,\cdot}\rangle$ is defined in (\ref{def:S}). The rightmost term
in (\ref{hi}) goes to zero in probability by Chebyshev's inequality.
%
We fix  $\lambda, \mu\in\R^2$ and $0\leq s< t\leq {\mathcal
T}$, and we define the following array of variables:
\begin{equation*}
\txi_{N,i}= \left\{
\begin{array}{lll}\vspace{0.4cm}
\langle \lambda,\bar \Psi_{N,\;i}\rangle & \mbox{if} & 0\leq i\leq \lfloor,
N s\rfloor -1,\\
\langle \mu,\bar \Psi_{N,\;i}\rangle & \mbox{if} &\lfloor N s\rfloor\leq i,\hspace{1.5cm}\forall N\geq 1.
\end{array}\right.
\end{equation*}
We denote with $\mathcal{F}_{N,i}$ the
$\sigma$-algebra generated by $(X_0,...,X_i)\times (e_0,..,e_i)$,
 $\forall N\geq 1$, $ i\geq 0$. Then 
$\{\txi_{N,i},\mathcal{F}_{N,i} \}_{N\geq 1,\; i\geq 0}$ is a
martingale difference array. In particular, since $\vert\langle \nu,\bar
\Psi_{N,\;i}\rangle\vert\leq 2(\ln N)^{-1/2}$ for every $i\geq
1$, for every unitary vector $\nu\in\R^2$, it follows that  $\forall \ve>0$, 
there exists $\bar N$ such that  $\vert \txi_{N,i}\vert <\ve$, $\forall
N\geq \bar N$, $\forall i\geq 1$. 
Therefore condition $(ii)$ is 
satisfied.

Moreover, with similar arguments of the proof of (\ref{eq:V1}), one
can prove that
\begin{equation*}
\sum_{i=1}^{\lfloor N t\rfloor} \E\left[\txi_{N,i}^2\vert
{\mathcal F}_{N, i-1} \right]\stackrel{P}{\to}\sigma^2\vert\lambda\vert^2
s + \sigma^2\vert\mu\vert^2 (t- s),
\end{equation*}
with $\sigma^2$ defined in (\ref{def:c}). Thus
\begin{equation*}\begin{split}
\sum_{i=1}^{\lfloor N s\rfloor-1} \langle \lambda,\bar\Psi_{N,\;i}\rangle + 
\sum_{i=\lfloor
N s\rfloor}^{\lfloor N t\rfloor-1}\langle \mu,\bar
\Psi_{N,\;i}\rangle = \sum_{i=1}^{\lfloor N t\rfloor}
\txi_{N,i}\\
\Rightarrow {\mathcal N}\big(0, \sigma^2\{\vert\lambda\vert^2 s +
\vert\mu\vert^2 (t-s)\}\big).
\end{split}\end{equation*}

\section{An invariance principle for centered,  bounded random variables}\label{sec:conv_mom}
In this section we present an alternative proof of Proposition \ref{prop:convZ_N}.
We start with
a CLT for arrays of centered, uniformly bounded random variables, based on
the convergence of the moments 
to the moments of a normal distribution. 
Some asymptotic factorization conditions, holding  on average, are required.
Then we will use it to show that 
for every unitary vector $\lambda\in\R^2$,
$\langle\lambda, Z^<_N(t) \rangle =\lambda_1 Z_N^{1<}(t)+\lambda_2 Z_N^{2<}(t)\Rightarrow W_\sigma(t)$,
$\forall t\in[0\mathcal T]$.
\begin{proposition}[CLT]
Let $\{\bar X _{n,i}\; i=1,..,n, n\geq 1\}$ be an array of centered
random variables and suppose that exists $\varepsilon_n\downarrow 0$
such that $\vert\bar X _{n,i}\vert \leq \varepsilon_n$, for all $n$
and $i$.  Let $\bar S_{n}=\sum_{i=1}^{n} \bar X _{n,i}$. Then $\bar
S_{n}\Rightarrow \mathcal{N}(0,c)$, if the following conditions
hold:
\begin{itemize}
\item[(i)]
$\forall \ell\geq 1$, for every sequence of positive integers
$\{p_1,..,p_\ell\}$ such that $\exists p_j=1$, $j\in\{1,..,\ell\}$
$$
\sum_{i_1\neq i_2\neq...\neq i_\ell}^n\E\left[(\bar X_{n,
i_1})^{p_1}... (\bar X_{n,
i_\ell})^{p_\ell}\right]\stackrel{n\uparrow\infty}{\longrightarrow
0}$$
\item[(ii)]
$\forall \ell\geq 1$
$$
\sum_{i_1\neq i_2\neq...\neq i_\ell}^n \E\left[(\bar X_{n,
i_1})^2... (\bar X_{n, i_\ell})^2\right]
\stackrel{n\uparrow\infty}{\longrightarrow c^\ell}
$$
\end{itemize}
\end{proposition}
\begin{proof}
The proof is based on the convergence of the moments of $\bar S_n$.
Of course $\E[\bar S_n]=0$, while for the second moment we have
$$
\E\left[(\bar S_n)^2\right]=\sum_{i=1}^n \E\left[(\bar X_{i,n})^2
\right]+\sum_{i\neq j}^n\E\left[\bar X_{i,n}\bar X_{j,n} \right]\to
c,
$$
since the second sum goes to zero for condition $(i)$.

Now let us compute the third moment:
$$
\E\left[(\bar S_n)^3\right]=\sum_{i=1}^n \E\left[(\bar X_{i,n})^3
\right]+3\sum_{i\neq j}^n\E\left[(\bar X_{i,n})^2\bar X_{j,n}
\right]+ \sum_{i\neq j\neq k}^n\E\left[\bar X_{i,n}\bar X_{j,n} \bar
X_{k,n} \right].
$$
The last two sums go to zero for condition $(i)$.  For the first sum
we have
$$
\left\vert \sum_{i=1}^n \E\left[(\bar X_{i,n})^3 \right]\right\vert
\leq \sum_{i=1}^n \E\left[(\bar X_{i,n})^2 \vert\bar X_{i,n}\vert
\right]\leq \varepsilon_n \sum_{i=1}^n \E\left[(\bar X_{i,n})^2
\right]\sim \varepsilon_n c\stackrel{n\to\infty}{\longrightarrow 0}.
$$

In the  general case,  the $m$-th moment $\E\left[(\bar
S_n)^m\right]$ is made up of terms of the form
$$
A(p_1,..,p_\ell)\sum_{i_1\neq i_2..\neq i_\ell}^n \E\left[(\bar
X_{i_1,n})^{p_1}...(\bar X_{i_\ell,n})^{p_\ell}\right], \hspace{1cm}
1\leq \ell\leq m$$ with $\{p_i,\;i=1,..,\ell\}$ positive integers
such that $p_1+p_2+..+p_\ell=m$.
 Here $A(p_1,..,p_\ell)$ is  the number
 of all possible partitions of $m$ objects in $\ell$ subsets made up
of $p_1,..,p_\ell$ objects.
 Since all sums containing a
singleton (i.e. there is a $p_i=1$) go asymptotically to zero, we
consider just the cases with $p_i\geq 2$, $\forall i=1,..,\ell$.
Observe that this implies in particular that $\ell\leq m/2$. In this
case
 \begin{equation*}\begin{split}
 \left\vert \sum_{i_1\neq i_2..\neq i_\ell}^n\E\left[(\bar
X_{i_1,n})^{p_1}...(\bar X_{i_\ell,n})^{p_\ell}\right]\right\vert
\leq \varepsilon_n^{m-2\ell}\sum_{i_1\neq i_2..\neq i_\ell}^n
\E\left[(\bar X_{i_1,n})^2...(\bar X_{i_\ell,n})^2 \right]\\
\sim \varepsilon_n^{m-2\ell} c^\ell,
 \end{split}\end{equation*}
 which goes to zero if $\ell\neq m/2$. Therefore all odd moments are
 asymptotically negligible, while for even moments asymptotically
 $$
 \E\left[(\bar S_n)^{2k}\right]\sim A_{k}\sum_{i_1\neq...\neq
 i_k}^n\E\left[(\bar X_{i_1,n})^2...(\bar X_{i_k,n})^2 \right]\to
 A_{k} c^k,
 $$
where $A_k$ is the number of all possible pairings of $2k$ objects,
namely
$$
A_k=(2k-1)(2k-3)\cdot\cdot\cdot 1=(2k-1)!!
$$
Finally
\begin{equation*}
  \E\left[(\bar S_n)^m\right]\stackrel{n\to\infty}{\longrightarrow}\left\{\begin{array}{ll}
                                 ×0 & m\; odd\\
                                  (m-1)!!\; c^{m/2} & m\;even,
                                \end{array}\right.
\end{equation*}
which are the moments of a Gaussian variable $\mathcal{N}(0,c)$.
\end{proof}

Let us consider the array of variables $\{\langle \lambda, \bar \Psi_{N,m}\rangle,  N\geq 2, m\geq 0\}$ defined in (\ref{def:lambdapsi}), (\ref{def1}),
with $\lambda\in \R^2$ unitary vector. 
We have
\begin{equation*}
\langle\lambda, Z^<_N(t) \rangle =\sum_{m=0}^{\lfloor Nt\rfloor -1}\langle \lambda, \bar \Psi_{N,m}\rangle
+ \big(Nt-\lfloor Nt\rfloor\langle \lambda, \bar \Psi_{N,\lfloor Nt\rfloor}\rangle \big),
\end{equation*}
$\forall t\in[0,\mathcal T]$, $\forall N\geq 2$, where the rightmost term goes to zero in probability by Chebyshev's inequality.
By definition, $\langle \lambda, \bar \Psi_{N,m}\rangle\leq \frac{2}{\sqrt{\ln N}}$ for every $m\geq 0$, $\forall N\geq 2$. 
Moreover, 
since $\psi(k)$ is an odd function, and the probability kernel $P(k,dk')$ has a density which is even in both  $k$ and $k'$, the array satisfies condition $(i)$.
In order to check condition $(ii)$, we will use the following Lemma.
\begin{lemma}\label{lemma(ii)}
For every $\ell\geq 1$, for every sequence $(m_1,....,m_\ell)$ such that $m_1\geq0$, $m_i\geq 1$, for every $N\geq 2$
\begin{equation}\label{bound(ii)}
\E\left[\langle \lambda, \bar \Psi_{N, m_1}\rangle^2.... \langle \lambda, \bar \Psi_{N,m_1+..+m_\ell}\rangle^2\right]\leq \frac{c_0^\ell}{N^\ell}
\end{equation}
with $c_0$ 
finite,
$\forall t\in[0,\mathcal T]$. 
\end{lemma}
\begin{proof}
By definition
\begin{equation*}\begin{split}
&
\E\left[\langle \lambda, \bar \Psi_{N, m_1}\rangle^2.... \langle \lambda, \bar \Psi_{N,m_1+..+m_\ell}\rangle^2\right]\\
&=
\int_0^\infty dz_1\; e^{-z_1}\int_{\Tt}\mu P^{m_1}(dk_1)\langle \lambda, \bar \Psi_{N}(k_1, z_1)\rangle^2\int...\\
&\;\;\times
\int_0^\infty dz_m\; e^{-z_m}\int_{\Tt}P^{m_\ell}(k_{m-1},k_m)\langle \lambda, \bar \Psi_{N}(k_m, z_m)\rangle^2\\
&\leq 2^\ell\left(\int_0^\infty dz\; e^{-z}\int_{\Tt}\pi(k)\langle \lambda, \bar \Psi_{N}(k,z)\rangle^2\right)^\ell,
\end{split}\end{equation*}
where in the last inequality we used  (\ref{ubP}). We conclude the proof by observing that
$$
\lim_{N\to\infty} N\;\int_0^\infty dz\; 
e^{-z}\int_{\Tt}\pi(k)\langle \lambda, \bar \Psi_{N}(k,z)\rangle^2=\sigma^2,
$$
with $\sigma$ defined in (\ref{def:c}).
\end{proof}
We observe that
\begin{equation*}\begin{split}
& \sum_{\genfrac{}{}{0pt}{}{i_1\neq i_2\neq...\neq i_\ell}{\in\{0,..,\lfloor Nt\rfloor-1\}}}\E\left[\langle \lambda, \bar \Psi_{N, i_1}\rangle^2.... \langle \lambda, \bar \Psi_{N,i_\ell}\rangle^2\right]\\
&=\ell! \sum_{m_1\geq 0}
\sum_{\genfrac{}{}{0pt}{}{m_2,..,m_\ell\geq 1}
       {m_1+...+m_\ell\leq \lfloor Nt\rfloor -1}}
\E\left[\langle \lambda, \bar \Psi_{N, m_1}\rangle^2.... \langle \lambda, \bar \Psi_{N,m_1+..+m_\ell}\rangle^2\right].
\end{split}\end{equation*}
We split the sum on $m_1$ in two part, namely $\sum_{m_1=0}^{M-1}+\sum_{m_1\geq M}$, with $0<M< \lfloor Nt\rfloor -1$. Using (\ref{bound(ii)}) and
the relation
$$
\lim_{N\to\infty}\sum_{\genfrac{}{}{0pt}{}{m_1,..,m_k\geq 1}
       {m_1+...+m_k\leq N}} N^{-k}=\frac{1}{k!},
$$ 
 we get that for every $\ell\geq 1$, $N\geq 2$, $\forall t\in[0,\mathcal T]$
\begin{equation*}
 \begin{split}
  \ell! \sum_{m_1=0}^{M-1}\sum_{\genfrac{}{}{0pt}{}{m_2,..,m_\ell\geq 1}
       {m_1+...+m_\ell\leq \lfloor Nt\rfloor -1}}
\E\left[\langle \lambda, \bar \Psi_{N, m_1}\rangle^2.... \langle \lambda, \bar \Psi_{N,m_1+..+m_\ell}\rangle^2\right]\\
\leq C_\ell \mathcal{T}^{\ell-1}\frac{M}{N}.
 \end{split}
\end{equation*}
By repeating this procedure for all the sums, we have
\begin{equation}\label{end2}\begin{split}
 & \sum_{\genfrac{}{}{0pt}{}{i_1\neq i_2\neq...\neq i_\ell}{\in\{0,..,\lfloor Nt\rfloor-1\}}}\E\left[\langle \lambda, \bar \Psi_{N, i_1}\rangle^2.... \langle \lambda, \bar \Psi_{N,i_\ell}\rangle^2\right]
\\
&=\ell! 
\sum_{\genfrac{}{}{0pt}{}{m_1,..,m_\ell\geq M}
       {m_1+...+m_\ell\leq \lfloor Nt\rfloor -1}}
\E\left[\langle \lambda, \bar \Psi_{N, m_1}\rangle^2.... \langle \lambda, \bar \Psi_{N,m_1+..+m_\ell}\rangle^2\right]+\mathcal{E}_\ell(M,N),
\end{split}\end{equation}
with $\mathcal{E}_\ell(M,N)\leq \tilde{C}_\ell \mathcal{T}^{\ell-1} M/N$, $\forall  \ell\geq 1$.

Observe that for every $m\geq 2$
\begin{equation*}\begin{split}
\int_{\Tt}  P^m(k,dk') \langle \lambda, \bar \Psi_{N}(k',z)\rangle^2
=\int_{\Tt}  \pi(dk') \langle \lambda, \bar \Psi_{N}(k',z)\rangle^2\\
+\int_{\Tt}  \left[P^{m-1}(k,d\tilde k)-\pi(d\tilde k)\right]\int_{\Tt} P(\tilde k,dk')\langle \lambda, \bar \Psi_{N}(k',z)\rangle^2,
\end{split}\end{equation*}
where, using (\ref{ubP}), 
\begin{equation*}\begin{split}
\sup_{k\in\Tt}  \int_{\Tt}  \left|P^{m-1}(k,d\tilde k)-\pi(d\tilde k)\right|\int_{\Tt} P(\tilde k,dk')\langle \lambda, \bar \Psi_{N}(k',z)\rangle^2\\
\leq \frac{C_0}{N}\sup_{k\in\Tt}  \int_{\Tt}  \left|P^{m-1}(k,d\tilde k)-\pi(d\tilde k)\right|.
\end{split}\end{equation*}
Thus, thanks to (\ref{bound(ii)}), for every $(m_1,..,m_\ell)$ with $m_i\geq M$, $i=1,..,\ell,$
\begin{equation}\label{end3}\begin{split}
& \E\left[\langle \lambda, \bar \Psi_{N, m_1}\rangle^2.... \langle \lambda, \bar \Psi_{N,m_1+..+m_\ell}\rangle^2\right]\\
&=\left(\int_0^\infty dz\; e^{-z} \int_{\Tt}  \pi(dk') \langle \lambda, \bar \Psi_{N}(k',z)\rangle^2\right)^\ell 
+\tilde{e}_\ell(M, N),
\end{split}
\end{equation}
where
$$
\tilde{e}_\ell(M, N)\leq \ell \frac{C_0}{N^\ell}\sup_{m\geq M-1}\sup_{k\in\Tt}  \int_{\Tt}  \left|P^{m}(k,d\tilde k)-\pi(d\tilde k)\right|.
$$
Finally, by (\ref{end2}) and (\ref{end3}) we get
\begin{equation*}
 \begin{split}
&\sum_{\genfrac{}{}{0pt}{}{i_1\neq i_2\neq...\neq i_\ell}{\in\{0,..,\lfloor Nt\rfloor-1\}}}\E\left[\langle \lambda, \bar \Psi_{N, i_1}\rangle^2.... \langle \lambda, \bar \Psi_{N,i_\ell}\rangle^2\right]
\\
&=\ell! 
\sum_{\genfrac{}{}{0pt}{}{m_1,..,m_\ell\geq M}
       {m_1+...+m_\ell\leq \lfloor Nt\rfloor -1}}
\left(\E_\pi\left[\langle \lambda, \bar \Psi_{N, 1}\rangle^2\right]\right)^\ell+\mathcal{R}_\ell(M,N),
\end{split}
\end{equation*}
where
\begin{equation}\label{def:Rl}
\mathcal{R}_\ell(M,N)\leq  C_\ell \mathcal{T}^{\ell}\left( \frac{M}{N}\; +\sup_{m\geq M-1}\sup_{k\in\Tt}  \int_{\Tt}  \left|P^{m}(k,d\tilde k)-\pi(d\tilde k)\right|\right).
\end{equation}
In the limit  $M,N\to\infty$ such that $\frac{M}{N}\to 0$, $\mathcal{R}_\ell(M,N)\to 0$ and
\begin{equation*}
 \ell! 
\sum_{\genfrac{}{}{0pt}{}{m_1,..,m_\ell\geq M}
       {m_1+...+m_\ell\leq \lfloor Nt\rfloor -1}}
\left(\E_\pi\left[\langle \lambda, \bar \Psi_{N, 1}\rangle^2\right]\right)^\ell\to {(\sigma^2)}^\ell t^\ell,
\end{equation*}
with $\sigma$ defined in $(\ref{def:c})$. Thus the array of variables $\{\langle \lambda, \bar \Psi_{N,m}\rangle,  N\geq 2, m\geq 0\}$
satisfies also condition $(ii)$, and we get
\begin{equation*}
 \bar{S}_N(t):=\sum_{n=0}^{\lfloor Nt\rfloor-1}\langle \lambda, \bar \Psi_{N,n}\rangle\stackrel{N\uparrow\infty}{\to}\mathcal{N}(0, \sigma^2\; t),
\end{equation*}
$\forall t\in[0,\mathcal T]$, $\forall \lambda\in\R^2$ such that $\vert \lambda \vert=1$.

We can easily adapt the proof and show that $\forall 0\leq s<t\leq \mathcal{T}$
$$\bar S_N(t)-\bar S_N(s)\to\mathcal N (0,\sigma^2\; (t-s)).$$
In order to prove the convergence of the finite dimensional 
marginal to the Wiener process $W_\sigma$, we have to show that
$\forall n\geq 2$, for every partition $0\leq t_1<...<t_n\leq \mathcal T$
the variables 
$\bar S_N(t_1)$,  $\bar S_N(t_2)-\bar S_N(t_1)$,..,$\bar S_N(t_n)-\bar S_N(t_{n-1})$ are asymptotically jointly Gaussian and uncorrelated. 
This is stated in the next Lemma.
\begin{lemma}
 For every $n\geq 1$,  
$\forall \;\underline{\alpha}(n):=(\alpha_1,..,\alpha_n) \in \R ^n$ such that 
$\vert \underline{\alpha}(n)\vert=1$ 
\begin{equation}
 \sum_{k=1}^n \alpha_k (\bar S_N(t_k)-\bar S_N(t_{k-1}))
 \Rightarrow \mathcal{N}\left( 0,\; \sigma^2\sum_{k=1}^n \alpha_k^2 (t_k-t_{k-1})\right),
\end{equation}
$\forall 0=t_0< t_1<..<t_n\leq \mathcal T$.
\end{lemma}
\begin{proof}
The case $n=1$ is proved.
 Let us consider the case $n=2$. 
Fixed $(\alpha_1,\alpha_2) \in\R^2$, with $\alpha_1^2+\alpha_2 ^2=1$, we consider the following array of variables
\begin{equation*}
 \xi_{N,m}=\left(\alpha_1\mathbf{1}_{\{m\leq \lfloor N t_1\rfloor -1\}}+ \alpha_2\mathbf{1}_{\{m\geq \lfloor N t_1\rfloor\}}\right)
\langle \lambda, \bar \Psi_{N, m}\rangle,\;\;\forall N\geq 2, \forall m\geq 0,
\end{equation*}
which are uniformly bounded by $\frac{2}{\sqrt N}$ and satisfy condition $(i)$. Let us define, $\forall t\geq 0$, $m\geq 0$, $N\geq 2$,
$$
a_{N,m}(t):=\alpha_1\mathbf{1}_{\{m\leq \lfloor N t\rfloor -1\}}+ \alpha_2\mathbf{1}_{\{m\geq \lfloor N t\rfloor\}},
$$
which is uniformly bounded by $1$. In order to check condition $(ii)$, we  repeat the steps done
 for  $\bar S_N (t)$ and we get
\begin{equation*}
\begin{split}
&\sum_{\genfrac{}{}{0pt}{}{i_1\neq i_2\neq...\neq i_\ell}{\in\{0,..,\lfloor Nt_2\rfloor-1\}}}\E\left[
\xi_{N,i_1}^2...\xi_{N,i_\ell}^2
\right]\\
&=\ell! \sum_{\scriptscriptstyle{ 0\leq i_1<..<i_\ell\leq \lfloor Nt_2\rfloor-1}}a_{N,i_1}(t_1)^2
...a_{N,i_\ell}(t_1)^2\left(\E_\pi\left[\langle \lambda, \bar \Psi_{N, 1}\rangle^2\right]\right)^\ell\\
&\;\;+\mathcal{R}_\ell(M,N),
\end{split}\end{equation*}
with $\mathcal{R}_\ell(M,N)$ the same of  (\ref{def:Rl}). By direct computation
\begin{equation*}\begin{split}
&      \ell! \sum_{\scriptscriptstyle{ 0\leq i_1<..<i_\ell\leq \lfloor Nt_2\rfloor-1}}a_{N,i_1}(t_1)^2
...a_{N,i_\ell}(t_1)^2     \\
& = \sum_{k=0}^\ell  \ell! \sum_{\scriptscriptstyle{ 1\leq i_1<..<i_k\leq \lfloor Nt_1\rfloor}}  (\alpha_1)^{2k} 
\sum_{\scriptscriptstyle{ \lfloor Nt_1\rfloor< i_{k+1}<..<i_\ell\leq \lfloor Nt_2\rfloor}}  (\alpha_2)^{2(\ell-k)},
                \end{split}
\end{equation*}
then using
\begin{equation*}
 \begin{split}
  \sum_{\scriptscriptstyle{ 1\leq i_1<..<i_k\leq  N}} N^{-k}\stackrel{N\uparrow\infty}{\to}\frac{1}{k!},\hspace{1cm}
N^\ell \left(\E_\pi\left[\langle \lambda, \bar \Psi_{N, 1}\rangle^2\right]\right)^\ell \stackrel{N\uparrow\infty}{\to}{(\sigma^2)}^\ell,
 \end{split}
\end{equation*}
with $\sigma$ defined in (\ref{def:c}), we get that condition $(ii)$ is satisfied, i.e.
\begin{equation*}\begin{split}
& \lim_{N\to\infty}
\sum_{\genfrac{}{}{0pt}{}{i_1\neq i_2\neq...\neq i_\ell}{\in\{0,..,\lfloor Nt_2\rfloor-1\}}}
\E\left[\xi_{N,i_1}^2...\xi_{N,i_\ell}^2\right] \\
&={(\sigma^2)}^\ell\sum_{k=0}^\ell\frac{\ell!}{k! (\ell-k)!}\;\alpha_1^{2k} \;t_1^k\; \alpha_2^{2(\ell-k)}\;(t_2-t_1)^{\ell-k}\\
&={(\sigma^2)}^\ell[\alpha_1^2\;t_1 +\alpha_2^2\;(t_2-t_1)]^\ell,
\end{split}
\end{equation*}
thus
\begin{equation*}\begin{split}
 \alpha_1 \bar S_N(t_1)+\alpha_2 [\bar S _N(t_2)-\bar S_N(t_1)]
\;=\sum_{m=0}^{\lfloor Nt_2\rfloor -1}\xi_{N,m}\\ \to \mathcal{N}(0,{(\sigma^2)}[\alpha_1^2\;t_1 +\alpha_2^2\;(t_2-t_1)] ).
\end{split}\end{equation*}
The proof can be repeated for $n\geq 3$, in that case we find the multinomial formula for a polynomial with $n$ terms to the power $\ell$.

\end{proof}


\end{document}